\documentclass[letterpaper,leqno,11pt,oneside]{amsart}

\usepackage{amsmath,amsthm,amsfonts,amssymb}
\usepackage{enumitem}
\usepackage[usenames,dvipsnames]{color}
\usepackage[numbers,longnamesfirst]{natbib}
\usepackage{mathtools,comment}
\usepackage{url}
\RequirePackage{hyperref}
\usepackage[height=9in,width=5.75in]{geometry}
\usepackage{graphics,graphicx}
\usepackage{mathrsfs}
 \usepackage{eucal}
 \usepackage{endnotes}

\definecolor{darkblue}{rgb}{0.03,0.03,0.265}
\hypersetup{colorlinks=true,urlcolor=darkblue,citecolor=darkblue,linkcolor=darkblue,pdftitle=Local
Brownian property of KPZ}

\mathtoolsset{showonlyrefs,showmanualtags}

\newtheorem{thm}{Theorem}
\newtheorem{lem}{Lemma}[section]
\newtheorem{prop}[lem]{Proposition}

\theoremstyle{definition}

\newtheorem*{rem*}{Remark}

\newcommand{\pp}{\mathbb{P}}

\newcommand{\ee}{\mathbb{E}}
\newcommand{\rr}{\mathbb{R}}
\newcommand{\nn}{\mathbb{N}}
\newcommand{\zz}{\mathbb{Z}}

\newcommand{\cz}{{\CMcal Z}}
\newcommand{\ch}{{\CMcal H}}
\newcommand{\cb}{{\CMcal B}}

\newcommand{\scw}{{\mathscr{W}}}

\newcommand{\p}{\partial}
\newcommand{\uno}[1]{\mathbf{1}_{#1}}
\newcommand{\heta}{\widehat{\eta}}
\newcommand{\etaep}{\eta^\ep}
\newcommand{\etaeq}{\eta^{\mathrm{eq}}}
\newcommand{\cheq}{\ch^{\mathrm{eq}}}

\newcommand{\ep}{\varepsilon}
\newcommand{\vs}{\vspace{6pt}}

\newcommand{\epp}[1][]{{\varepsilon{}^{#1}}}
\DeclareMathOperator{\TV}{TV}

\numberwithin{equation}{section}

\let\oldmarginpar\marginpar
\renewcommand\marginpar[1]{\-\oldmarginpar[\raggedleft\footnotesize #1]%
{\raggedright{\small\textsf{#1}}}}

\begin{document}

\title[Local Brownian property  of KPZ]{Local Brownian property of the narrow wedge
  solution of the KPZ equation}
\author{Jeremy Quastel}
\address[J.~Quastel]{
  Department of Mathematics\\
   University of Toronto\\
   40 St. George Street\\
   Toronto, Ontario\\
   Canada M5S 2E4}
 \email{quastel@math.toronto.edu}
\author{Daniel Remenik}
\address[D.~Remenik]{
  Department of Mathematics\\
  University of Toronto\\
  40 St. George Street\\
  Toronto, Ontario\\
  Canada M5S 2E4
  \newline
  \indent\textup{and}\indent
  Departamento de Ingenier\'ia Matem\'atica\\
  Universidad de Chile\\
  Av. Blanco Encala\-da 2120\\
  Santiago\\
  Chile}
  \email{dremenik@math.toronto.edu}

\begin{abstract}
  Let $\ch(t,x)$ be the Hopf-Cole solution at time $t$ of the Kardar-Parisi-Zhang (KPZ) equation
  starting with narrow wedge initial condition, i.e. the logarithm of the solution of the
  multiplicative stochastic heat equation starting from a Dirac delta. Also let
  $\cheq(t,x)$ be the solution at time $t$ of the KPZ equation with the same noise, but with initial condition
  given by a standard two-sided Brownian motion, so that $\cheq(t,x)-\cheq(0,x)$ is itself
  distributed as a standard two-sided Brownian motion. We provide a simple proof of the
  following fact: for fixed $t$, $\ch(t,x)-\big(\cheq(t,x)-\cheq(t,0)\big)$ is locally of  finite
  variation.  Using the same ideas we also show that if the KPZ equation is
  started with a two-sided Brownian motion plus a Lipschitz function then the solution
  stays in this class for all time.
\end{abstract}

\maketitle

\section{Introduction and statement of the results}

The \emph{KPZ equation}
\begin{equation}
  \label{eq:kpz}
  \p_t\ch=-\frac{1}{2}\big(\p_x\ch\big)^2+\frac{1}{2}\p_x^2\ch+\dot \scw,
\end{equation}
was introduced by \citet{KPZ} as a model of randomly growing interfaces. Here
$\dot\scw(t,x)$ is Gaussian space-time white noise,
$\ee\big(\dot\scw(t,x)\scw(s,y)\big)=\delta_{s=t}\delta_{x=y}$ (see Section 1.4 of
\cite{acq} for a precise definition). It is expected that the one-dimensional KPZ equation
appears as the weak asymptotic limit of a large class of stochastic interacting particle
systems/growth models, including directed random polymers, stochastic
Hamilton-Jacobi-Bellman equations, stochastically perturbed reaction-diffusion equations,
stochastic Burgers equations and interacting particle models, and it is in fact
rigourously known to describe the fluctuations in weakly asymmetric exclusion processes
\cite{berGiaco,acq,corQuas} and the partition function in directed polymer models
\cite{acq,akq,mqr2}. All these models belong to the so-called \emph{KPZ universality
  class}, which is associated with unusual fluctuations of order $t^{1/3}$ at time $t$ on
a spatial scale of $t^{2/3}$. We refer the reader to the reviews \cite{corwin,quastel} for
more details and background on the KPZ equation and universality class.

As stated the KPZ equation \eqref{eq:kpz} is ill-posed due to the non-linear term.  To
make sense of it we follow the approach of \citet{berGiaco}. Observe that if we let
$\cz(t,x)=\exp(-\ch(t,x))$ then, formally, $\cz(t,x)$ solves the (linear) stochastic heat
equation with multiplicative noise
\begin{equation}
  \label{eq:heat}
  \p_t\cz=\frac{1}{2}\p^2_x\cz-\cz\dot \scw.
\end{equation}
Therefore we simply define the solutions $\ch(t,x)$ of \eqref{eq:kpz} via the Hopf-Cole
transformation
\begin{equation}
  \label{eq:hopfCole}
  \ch(t,x)=-\log(\cz(t,x)),
\end{equation}
where $\cz(t,x)$ is the (well-defined) solution of the stochastic PDE \eqref{eq:heat}. In a remarkable recent development, M. Hairer \cite{hairer} has
proposed a  way to make sense of the KPZ equation directly. The resulting solutions  coincide with the Hopf-Cole solutions.

One of the most interesting properties of the KPZ equation is the preservation of Brownian
initial data.  In particular, if one starts the equation with a standard two-sided
Brownian motion, one sees at time $t$ a new Brownian motion with the
same diffusivity, but with a (random) height shift (the new Brownian motion will of course
be coupled to the starting one in a highly non-trivial way). Furthermore, any initial
data, however smooth, will immediately become locally Brownian. This can be understood in
many ways.  One is that one expects the local quadratic variation to be the same as that
of the equilibrium solutions, for any positive time, for arbitrarily nice initial data.
Another is that one expects that the solution at time $t$ can be written as a standard
two-sided Brownian motion $\cb(x)$ plus a more regular object.  One would naturally like
to take this Brownian motion $\cb(x)$ to be the solution of the equation starting from a
two-sided Brownian motion. In other words, one would like to couple all solutions to the
equilibrium one.  For a large class of initial data, Hairer \cite{hairer} has shown that
the solution can be written as a Brownian motion plus a function in ${\CMcal
  C}^{\frac{3}{2}-}$.  Unfortunately, the Brownian motion used is a solution of the
Langevin equation obtained by linearizing KPZ, as opposed to the equilibrium solution of
KPZ itself.  Moreover, not all initial data can be handled by the methods of \cite{hairer}
because they require that certain auxiliary objects be integrable against heat kernels in
space and time.  The singularity at time $0$ rules out one of the most important cases,
which is the narrow wedge initial data.

Our main interest will be this last case: the initial condition for KPZ given by starting
the stochastic heat equation \eqref{eq:heat} with initial data
\begin{equation}
\cz(0,x)=\delta_{x=0}.\label{eq:delta}
\end{equation}
This defines $\ch(t,x)$ for
every $t>0$ via \eqref{eq:hopfCole}, but one should not think in terms of the initial data
for $\ch$, since the delta function does not have a well-defined logarithm.  This
\emph{narrow wedge} initial data is very basic.  For
example, it is the one that approximates the free energy of point-to-point polymers.

For this initial data, if we define
${\CMcal A}_t(x)$ by
\[\ch(t,x)=\frac{x^2}{2t}+\log\!\big(\sqrt{2\pi t}\big)+\frac{t}{24}-2^{-1/3}t^{1/3}{\CMcal A}_t(2^{-1/3}t^{-2/3}x)\]
then, properly rescaled, ${\CMcal A}_t(x)$ converges to a Gaussian process as $t\to0$ and 
it is conjectured that it converges to
the Airy$_2$ process as $t\to\infty$ (see Conjecture 1.5 in \cite{acq}). ${\CMcal A}_t(x)$
is therefore referred to as the \emph{crossover Airy$_2$ process}, and interpolates
between the KPZ and Edwards-Wilkinson \cite{edWilk} universality classes, the last one
associated with the stochastic heat equation with additive noise and hence Gaussian
statistics.

We will denote by $\cheq(t,x)$ the solution of the KPZ equation \eqref{eq:kpz} started
with initial condition
\begin{equation}
\cheq(0,x)=\cb(x),\label{eq:eqic}
\end{equation}
where $\cb(x)$ is a two-sided standard Brownian motion. We recall that this initial
condition is such that, for each fixed $t\geq0$, $\cheq(t,x)-\cheq(t,0)$ is itself a
two-sided standard Brownian motion (in space), see Proposition B.2 in \cite{berGiaco}.  In
this equation we use the same white noise as in the earlier solution starting with Dirac
mass (\ref{eq:delta}).  The two solutions exist, are unique, and are coupled for all
time.


Our main result is the following:

\begin{thm}\label{thm:totvar}
  Fix $t>0$ and let $\ch(t,x)$ and $\cheq(t,x)$ be the Hopf-Cole solutions of the KPZ
  equation \eqref{eq:kpz} with respect to the same white noise and with initial conditions
  given by \eqref{eq:delta} and \eqref{eq:eqic}. Then
  $\ch(t,x)-\big(\cheq(t,x)-\cheq(t,0)\big)$ is a finite variation process.
\end{thm}

  
  Our initial data is special, but
it is in some sense the furthest possible from equilibrium, and it has the benefit of a
surprisingly simple proof.

We remark that if $\ch(t,x)$ is started with initial condition given by a two-sided
Brownian motion plus a Lipschitz function then it is easy to show using our coupling
method that it remains a two-sided Brownian motion plus a Lipschitz function for all
$t>0$. This follows from the results of Hairer \cite{hairer}, but the proof there is much
more involved. The precise statement is given next, its short proof uses the same ideas
as the proof of Theorem \ref{thm:totvar} and is given in Section \ref{sec:proof}.

\begin{thm}\label{thm:neareq}
  Let $\cheq(t,x)$ and $\ch(t,x)$ be the Hopf-Cole solutions of the KPZ equation
  \eqref{eq:kpz} with respect to the same white noise and with initial conditions given
  respectively by \eqref{eq:eqic} and
  \[\ch(0,x)=\cb(x)+\varphi(x),\]
  where $\cb(x)$ is the same two-sided standard Brownian motion as in \eqref{eq:eqic} and
  $\varphi$ is a Lipschitz function. Then, for every fixed $t\geq0$,
  $\ch(t,x)-\big(\cheq(t,x)-\cheq(t,0)\big)$ is almost surely a Lipschitz function (with
  the same Lipschitz constant as $\varphi$). In particular, the law of $\ch(t,x)$ in a
  finite interval has finite relative entropy with respect to the law of $\cb(x)$ in that
  interval.
\end{thm}

\section{Proofs}\label{sec:proof}

The proofs of Theorems \ref{thm:totvar} and \ref{thm:neareq} rely on considering the weakly
asymmetric simple exclusion process, which provides a microscopic model for the KPZ
process. The \emph{simple exclusion process} with parameters $p,q\in[0,1]$ (such that
$p+q=1$) is a $\{0,1\}^\zz$-valued continuous time Markov process, where 1's are thought
of as particles and 0's as holes. The dynamics of the process are as follows: each
particle has an independent exponential clock with parameter 1; when the clock rings, the
particle attempts a jump, trying to go one step to the right with probability $p$ and one
step to the left with probability $q$; if there is a particle at the chosen destination,
the jump is supressed and the clock is reset. We refer the reader to \cite{liggett} for a
rigorous construction of this process. We will be interested in the case $q>p$, known as
the \emph{asymmetric simple exclusion process} (ASEP). More precisely, we will be
interested in the \emph{weakly asymmetric simple exclusion process} (WASEP), where we
introduce a parameter in the model and let the asymmetry $q-p$ go to 0 with the parameter.

Given any configuration $\eta\in\{0,1\}^\zz$ for the exclusion process we will denote by
$\heta\in\{-1,1\}$ the configuration given by $\heta(x)=2\eta(x)-1$ for each $x\in\zz$.
To any simple exclusion process $\eta_t$ we can associate the \emph{height function}
$h(t,\cdot)\!:\rr\longrightarrow\zz$ in the following manner:
\begin{equation}
  \label{eq:height}
  h(t,x) = \begin{dcases*}
    2N(t) + \sum_{0<y\leq x}\heta_t(y) & if $x>0$,\\
    2N(t) & if $x=0$,\\
    2N(t)-\sum_{x<y\leq 0}\heta_t(y) & if $x<0$,
  \end{dcases*}
\end{equation}
where $N(t)$ is the net number of particles which crossed from the site 1 to the site 0 up
to time $t$. It is straightforward to check that the simple exclusion process can be
recovered from the height function by
\begin{equation}
  \heta_t(x)=h(t,x)-h(t,x-1).\label{eq:etaFromHeight}
\end{equation}

Next we introduce the scaling parameter $\ep>0$, which should be thought of as going to 0,
and consider WASEP with asymmetry $\epp[1/2]$, that is,
\[p-q=\epp[1/2],\qquad p=\tfrac{1}{2}-\tfrac{1}{2}\epp[1/2], \qquad
q=\tfrac{1}{2}+\tfrac{1}{2}\epp[1/2].\] We will denote by $\etaep_t$ the resulting WASEP,
which we start with the step initial condition
\begin{equation}
\eta_0(x)=\uno{x\geq0}.\label{eq:step}
\end{equation}
To $\etaep_t$ we associate the \emph{height function} $h_\ep(t,x)$ via
\eqref{eq:height}. Our main tool will be the convergence of a suitably rescaled version of
$h_\ep$ to the solution of the KPZ equation with initial condition \eqref{eq:delta}.

The convergence of the height function was proved by \citet*{acq} by performing a microscopic Hopf-Cole
transform analogous to \eqref{eq:hopfCole}, an idea introduced originally by
\citet{gartner} and further developed in \cite{berGiaco}. Let
\begin{equation}\label{eq:scaling}
  \begin{gathered}
    \gamma_\ep=\tfrac12\epp[-1/2],\qquad
    \lambda_\ep=\tfrac{1}{2}\log(\tfrac{p}{q})=\epp[1/2]+\tfrac{1}{3}\epp[3/2]+O(\epp[5/2]),\\
    v_\ep=p+q-2\sqrt{pq}=\tfrac{1}{2}\epp[1/2]+\tfrac{1}{8}\epp[3/2]+O(\epp[5/2]).
  \end{gathered}
\end{equation}
The \emph{Hopf-Cole transformed height function} is given by
\begin{equation}
  \label{eq:hcHeight}
  Z_\ep(t,x)=\gamma_\ep\exp\!\Big(-\lambda_\ep h_\ep(\epp[-2]t,\epp[-1]x)+v_\ep t\Big).
\end{equation}
Observe that, with this definition, $Z_\ep(0,x)\rightarrow\delta_{x=0}$ as $\ep\to0$
as discussed in Section 1.2 of \cite{acq}.

We regard the process $Z_\ep(t,x)$ as taking values in the space $D([0,\infty),D_u(\rr))$,
where $D_u(\rr)$ refers to right-continuous paths with left limits with the topology of
uniform convergence on compact sets, which we endow with the Skorohod topology. The
following result corresponds to Theorem 1.14 of \cite{acq}:

\begin{thm}\label{thm:cvgce}
  The family of processes $\big(Z_\ep\big)_{\ep>0}$ converges in distribution as $\ep\to0$
  in $D([0,\infty),D_u(\rr))$ to the $C([0,\infty),C(\rr))$-valued process $\cz$ given by
  the solution of the stochastic heat equation \eqref{eq:heat} with initial condition
  $\cz(0,x)=\delta_{x=0}$.
\end{thm}

We recall that the exclusion process is \emph{attractive}, which for our purposes means
that two copies $\eta^1_t$ and $\eta^2_t$ of the process with initial conditions
$\eta^1_0\leq\eta^2_0$ (which just means $\eta^1_0(x)\leq\eta^2_0(x)$ for all $x$) can be
coupled in such a way that $\eta^1_t\leq\eta^2_t$ for all $t>0$. We will refer to this
coupling as the \emph{basic coupling} and refer the reader to \cite{liggett} for more
details.

We will denote by $\etaeq_t$ a copy of WASEP in equilibrium, started with a product
measure with density $\frac{1}{2}$, and by $h^\mathrm{eq}$ the associated height
function. To prove Theorem \ref{thm:totvar} we will couple $\etaep_t$ with $\etaeq_t$
using the basic coupling. The key result will be an estimate on the number of
discrepancies between the two processes at time $\epp[-2]t$ in a window of size
$O(\epp[-1])$, Proposition \ref{prop:discr} below.

Let $\eta^{\min}_t$ and $\eta^{\max}_t$ denote copies of WASEP started with
initial conditions
\[\eta^{\min}_0=\etaep_0\wedge\etaeq_0\qquad\text{and}\qquad
\eta_0^{\max}=\etaep_0(x)\vee\etaeq_0,\] where the minimum and maximum are meant
sitewise. Observe that $\eta_0^{\min}$ corresponds to starting with no particles on the
negative half-line and a product measure of density $\frac{1}{2}$ on the positive
half-line, while $\eta^{\max}_0$ corresponds to starting with a product measure of density
$\frac{1}{2}$ on the negative half-line and all sites occupied on the positive
half-line. We will denote by $h^{\min}_\ep$ and $h^{\max}_\ep$ the height functions associated
respectively to these two processes.

Let $Z_\ep^{\min}$, $Z_\ep^{\max}$ and $Z^\mathrm{eq}_\ep$ be the Hole-Copf transformed
height functions associated to the corresponding initial conditions, which are defined in
the same way as $Z_\ep$ in \eqref{eq:hcHeight} with the scaling \eqref{eq:scaling} except
that $\gamma_\ep=1$. The proof in \cite{acq} of Theorem \ref{thm:cvgce} can be adapted
without difficulty (see \cite{corQuas} for the details) to show that that $Z_\ep^{\min}$
and $Z_\ep^{\max}$ converge in distribution in $D([0,\infty),D_u(\rr))$ respectively to
the solutions $\cz^{\min}(t,x)$ and $\cz^{\max}(t,x)$ of the stochastic heat equation
\eqref{eq:heat} with initial data $\cz^{\min}(0,x)=\exp\!\big(\!-\cb(x)\big)\uno{x\geq0}$ and
$\cz^{\max}(0,x)=\exp\!\big(\!-\cb(-x)\big)\uno{x<0}$, where $\cb(x)$ is a standard one-sided
Brownian motion. We define $\ch^{\min}(t,x)=-\log(\cz^{\min}(t,x))$ and
$\ch^{\max}(t,x)=-\log(\cz^{\max}(t,x))$.

Given any of the height functions $h$ with the different initial
conditions we are considering, we will denote by $\tilde h_\ep$ its rescaled version
\[\tilde h_\ep(t,x)=\epp[1/2]h(\epp[-2]t,\epp[-1]x).\]

\begin{prop}\label{prop:discr}
  Assume $\eta^\ep_t$ is started with the step initial condition \eqref{eq:step} and fix
  $a<b$ and $t>0$. Then, under the basic coupling,
  \begin{multline}
    \epp[1/2]\sum_{x\in[a\epp[-1],b\epp[-1]]\cap\zz}
    \left|\eta^\ep_{\epp[-2]t}(x)-\etaeq_{\epp[-2]t}(x)\right|\\
    \leq\frac{1}{2}\big[\tilde h_\ep^{\max}(t,b)-\tilde h_\ep^{\max}(t,a)\big]
    -\frac{1}{2}\big[\tilde h_\ep^{\min}(t,b)-\tilde h_\ep^{\min}(t,a)\big]
  \end{multline}
 almost surely. 
\end{prop}

\begin{proof}
  We construct the four processes $\etaep_t$, $\etaeq_t$, $\eta^{\min}_t$ and
  $\eta^{\max}_t$ together under the basic coupling, so attractiveness implies
  that
  \[\eta_t^{\min}\leq\etaep_t\wedge\etaeq_t\leq\etaep_t\vee\etaeq_t\leq\eta^{\max}_t\]
  for all $t>0$. Using this we get by \eqref{eq:etaFromHeight} that
  \begin{align}
    &\sum_{x\in[a\epp[-1],b\epp[-1]]\cap\zz}
    \Big|\etaep_{\epp[-2]t}(x)-\etaeq_{\epp[-2]t}(x)\Big|\\
    &\hspace{0.5in}=
    \sum_{x\in[a\epp[-1],b\epp[-1]]\cap\zz}
    \big[\etaep_{\epp[-2]t}(x)\vee\etaeq_{\epp[-2]t}(x)-\etaep_{\epp[-2]t}(x)\wedge\etaeq_{\epp[-2]t}(x)\big]\\
    &\hspace{0.5in}\leq\sum_{x\in[a\epp[-1],b\epp[-1]]\cap\zz}
    \left[\eta^{\max}_{\epp[-2]t}(x)-\eta^{\min}_{\epp[-2]t}(x)\right]
    =\frac{1}{2}\sum_{x\in[a\epp[-1],b\epp[-1]]\cap\zz}
    \left[\widehat{\eta}^{\max}_{\epp[-2]t}(x)-\widehat{\eta}^{\min}_{\epp[-2]t}(x)\right]\\
    &\hspace{0.5in}=\frac{1}{2}\sum_{x\in[a\epp[-1],b\epp[-1]]\cap\zz}
    \bigg[\Big(h^{\max}_\ep(\epp[-2]t,x)-h^{\max}_\ep(\epp[-2]t,x-1)\Big)\\
      &\hspace{3in}-\Big(h^{\min}_\ep(\epp[-2]t,x)-h^{\min}_\ep(\epp[-2]t,x-1)\Big)\bigg]\\
    &\hspace{0.5in}=\frac{\epp[-1/2]}{2}\big[\tilde h^{\max}_\ep(t,b)-\tilde h^{\max}_\ep(t,a)\big]
    -\frac{\epp[-1/2]}{2}\big[\tilde h^{\min}_\ep(t,b)-\tilde h^{\min}_\ep(t,a)\big].
  \end{align}
  Multiplying by $\epp[1/2]$ we obtain the desired bound.
\end{proof}

\begin{proof}[Proof of Theorem \ref{thm:totvar}]
  Fix a finite interval $I=[a,b]$ and let $\TV_I(f)$ denote the total variation of $f$ in
  $I$:
  \[\TV_I(f)=\sup_{a=x_0<x_1<\dotsm<x_n=b,\,n\in\nn}\,\sum_{i=1}^n|f(x_i)-f(x_{i-1})|.\]
  Then clearly
  \[\TV_I(\tilde h_\ep(t,\cdot)-\tilde h^\mathrm{eq}_\ep(t,\cdot))
  =2\epp[1/2]\sum_{x\in[a\epp[-1],b\epp[-1]]\cap\zz}
  \left|\eta^\ep_{\epp[-2]t}(x)-\etaeq_{\epp[-2]t}(x)\right|,\] so by Proposition
  \ref{prop:discr} we get
  \[\TV_I(\tilde h_\ep(t,\cdot)-\tilde h^\mathrm{eq}_\ep(t,\cdot))
  \leq\big[\tilde h^{\max}_\ep(t,b)-\tilde h^{\max}_\ep(t,a)\big] -\big[\tilde
  h^{\min}_\ep(t,b)-\tilde h^{\min}_\ep(t,a)\big].\] On the other hand
  \[\tilde h_\ep(t,x)-\tilde h^\mathrm{eq}_\ep(t,x)=-\epp[1/2]\lambda_\ep^{-1}\big[
  \log(Z_\ep(t,x))-\log(Z^\mathrm{eq}_\ep(t,x))]+\epp[-1/2]\lambda_\ep^{-1}\log(\gamma_\ep).\]
  Thus by Theorem \ref{thm:cvgce} and Theorem 2.3 in \cite{berGiaco} we get that $\tilde
  h_\ep(t,x)-\tilde h^\mathrm{eq}_\ep(t,x)-\epp[-1/2]\lambda_\ep^{-1}\log(\gamma_\ep)$
  converges in distribution to $\ch(t,x)-\cheq(t,x)$ on the interval $I$. Note that this
  requires a very minor extension of the results of \cite{acq} and \cite{berGiaco}, namely
  that the processes $\tilde h_\ep$ and $\tilde h^\mathrm{eq}_\ep$ built from exclusion
  processes running with the same background Poisson processes converge jointly to $\ch$
  and $\ch^\mathrm{eq}$. There are no issues involved in extending Theorem \ref{thm:cvgce}
  and Theorem 2.3 in \cite{berGiaco} to this situation and therefore we omit the details.

  By the lower semicontinuity of $\TV_I$ we deduce that
  \begin{align}
    &\pp\!\left(\TV_I\!\big(\ch(t,\cdot)-\big[\cheq(t,\cdot)-\cheq(t,0)\big]\big)>K\right)
    =\pp\!\left(\TV_I\!\big(\ch(t,\cdot)-\cheq(t,\cdot)\big)>K\right)\\
    &\qquad\qquad\leq\limsup_{\ep\to0}\pp\!\left(\TV_I\!\left(\tilde
        h_\ep(t,\cdot)-\tilde h^\mathrm{eq}_\ep(t,\cdot)-\epp[-1/2]\lambda_\ep^{-1}\log(\gamma_\ep)\right)>K\right)\\
    &\qquad\qquad=\limsup_{\ep\to0}\pp\!\left(\TV_I\!\left(\tilde h_\ep(t,\cdot)-\tilde
        h^\mathrm{eq}_\ep(t,\cdot)\right)>K\right)\\
    &\qquad\qquad\leq\limsup_{\ep\to0}\pp\!\left(\big(\tilde h^{\max}_\ep(t,b)-\tilde
      h^{\max}_\ep(t,a)\big) -\big(\tilde h^{\min}_\ep(t,b)-\tilde
      h^{\min}_\ep(t,a)\big)>K\right).
  \end{align}
  To finish the proof of Theorem \ref{thm:totvar} we observe that the quantity inside the
  last probability above equals
  \[-\epp[1/2]\lambda_\ep^{-1}\big[\log(Z_\ep^{\max}(t,b))-\log(Z_\ep^{\max}(t,a))
  -\log(Z_\ep^{\min}(t,b))+\log(Z_\ep^{\min}(t,a))\big].\] Note the key point that the
  additive constants in \eqref{eq:hcHeight} cancel. Using the convergence of
  $Z_\ep^{\min}$ and $Z_\ep^{\max}$ to solutions of the stochastic heat equation discussed
  before Proposition \ref{prop:discr}, the above converges in distribution to
  $\big[\ch^{\max}(t,b)-\ch^{\max}(t,a)\big]-\big[\ch^{\min}(t,b)-\ch^{\min}(t,a)\big]$.
  Since the last random variable is finite we deduce that
  \[\lim_{K\to\infty}\pp\!\left(\TV_I\!\big(\ch(t,\cdot)-\big[\cheq(t,\cdot)-\cheq(t,0)\big]\big)>K\right)=0.
  \qedhere\]
\end{proof}

\begin{proof}[Proof of Theorem \ref{thm:neareq}]
  Let $\eta_t$ be a copy of WASEP started with product measure with density profile given
  by
  \[\pp(\eta_0(x+1)=1)=\frac12+\frac12\ep^{-1/2}\left(\varphi(\ep
    x)-\varphi(\ep(x-1))\right).\] Then Theorem 2.3 of \cite{berGiaco} implies that
  $\tilde h_\ep(t,x)$ converges in distribution as $\ep\to0$ in $D([0,\infty),D_u(\rr))$
  to $\ch(t,x)$ (with initial condition as in the statement of the theorem). Now denote by
  $M$ the Lipschitz constant of $\varphi$ and let $\tilde h^+_\ep$ and $\tilde h^-_\ep$
  denote the rescaled height functions corresponding to WASEP started respectively with
  product measures of densities $\frac12(1+\epp[1/2]M)$ and
  $\frac12(1-\epp[1/2]M)$. Coupling the initial conditions in the natural way and using
  the basic coupling and attractiveness, it is clear that $\tilde{h}^-_\ep(t,x)\leq\tilde
  h_\ep(t,x)\leq\tilde{h}^+_\ep(t,x)$ for all $t>0$. On the other hand, since product
  measures are invariant for WASEP, $\tilde h^\pm_\ep(t,x)$ converges in distribution to
  $\cheq(t,x)\pm Mx$, and as before this convergence can be achieved jointly for $\tilde
  h_\ep$, $\tilde h^+_\ep$ and $\tilde h^-_\ep$. Therefore, given any $a<b$,
  $\cheq(t,x)-Mx\leq\ch(t,x)\leq\cheq(t,x)+Mx$ almost surely for every $x\in[a,b]$, and
  the result follows.
\end{proof}

\vs
\vs
\paragraph{\bf Acknowledgments}

Both authors were supported by the Natural Science and Engineering Research Council of Canada, and the second 
author was supported by a Fields-Ontario Postdoctoral Fellowship.  Part of this work was done during the Fields Institute 
program ``Dynamics and Transport in Disordered Systems" and the authors would like to thank the Fields Institute for 
its hospitality.

\bibliographystyle{natbib} \bibliography{../biblio}

\end{document}